\documentclass[preprint,12pt]{elsarticle}
\usepackage{tikz}

\usetikzlibrary{arrows.meta}
\usetikzlibrary{arrows}

\usetikzlibrary{patterns}
\usetikzlibrary{decorations.markings}

\tikzset{sno/.style={circle, draw, fill=black!50,inner sep=0pt, minimum width=4pt}}

\usepackage{amsthm,amsfonts,amssymb,amscd,amsmath,enumerate,verbatim,calc,graphicx,geometry}
\usepackage[all]{xy}
\newtheorem{theorem}{Theorem}[section]
\newtheorem{lemma}[theorem]{Lemma}
\newtheorem{proposition}[theorem]{Proposition}
\newtheorem{corollary}[theorem]{Corollary}
\theoremstyle{definition}
\theoremstyle{definitions}
\newtheorem{definition}[theorem]{Definition}

\newtheorem{remark}[theorem]{Remark}
\newtheorem{example}[theorem]{Example}
\theoremstyle{notations}

\theoremstyle{remarks}

\journal{}

\begin{document}

\begin{frontmatter}



\title{On Subgroup Topologies on Fundamental Groups}


\author[]{M. Abdullahi Rashid}
\ead{mbinev@mail.um.ac.ir}
\author[]{N. Jamali}
\ead{no.jamali@stu.um.ac.ir}
\author[]{B. Mashayekhy}
\ead{bmashf@um.ac.ir }
\author[]{S.Z. Pashaei}
\ead{Pashaei.seyyedzeynal@stu.um.ac.ir}
\author[]{H. Torabi\corref{cor1}}
\ead{h.torabi@ferdowsi.um.ac.ir}

\address{Department of Pure Mathematics, Center of Excellence in Analysis on Algebraic Structures, Ferdowsi University of Mashhad,\\
P.O.Box 1159-91775, Mashhad, Iran.}
\cortext[cor1]{Corresponding author}

\begin{abstract}
It is important to classify covering subgroups of the fundamental group of a topological space using their topological properties in the topologized fundamental  group. In this paper, we introduce and study some topologies on the fundamental group and use them to classify coverings, semicoverings, and generalized coverings of a topological space. To do this, we use the concept of subgroup topology on a group and discuss their properties. In particular, we explore which of these topologies make the fundamental group a topological group. Moreover, we provide some examples of topological spaces to compare topologies of fundamental groups.

\end{abstract}

\begin{keyword}
Semicovering \sep Generalized covering\sep Topological group\sep Spanier topology\sep Whisker topology\sep Lasso topology \sep Subgroup topology.

\MSC[2010]{57M10, 57M12, 57M05, 55Q05.}

\end{keyword}

\end{frontmatter}



\section{Introduction and Motivation}

The concept of a natural topology for the fundamental group is introduced by Hurewicz \cite{Hur} in 1935. It received further attention in 1950 by Dugundji \cite{Dug} and more recently by Biss \cite{Biss}, Calcut and McCarthy \cite{CalMc}, Brazas \cite{BrazT}  and others. For instance, Calcut and McCarthy  proved the following theorem.

\begin{theorem} \cite{CalMc} \label{tr11}
Let $ X $ be a locally path connected topological space. The topological fundamental group $ \pi_1^{qtop}(X,x) $ is discrete if and only if $ X $ is semilocally simply connected.
\end{theorem}

It is known that out of the category of semilocally simply connected spaces, classification of covering spaces is not accessible. Brazas \cite{BrazSO} showed that for semicovering spaces by some nice local properties, there is a classification based on the qtop-topology on the fundamental group. The purpose of this paper is to introduce and study some other topologies on the fundamental group to provide a classification of covering, semicovering and generalized coverings of a topological space. In addition, similar to Theorem \ref{tr11} it is of interest to find out for which topological space, the relative topologized fundamental group is discrete or trivial under new topologies (see the diagram).

 
 Recall that a continuous map $ p: \widetilde{X} \rightarrow X $ is a covering map if every point of $ X  $ has an open neighborhood which is evenly covered by $ p $. Brazas \cite{BrazS} defined semicoverings by removing the evenly covered property and keeping local homeomorphism with continuous lifting of paths and homotopies. Based on some simplifications done in \cite{BrazSO, Kow2}, we use a continuous surjective local homeomorphism with the unique path lifting property as the standard definition of semicovering maps. For generalized coverings, the local homeomorphism is replaced with the unique lifting property (see \cite{paper1}). In each case, the induced homomorphism $ p_*:\pi_1(\widetilde{X},\tilde{x}) \rightarrow  \pi_1(X,x) $ is a monomorphism and so $  \pi_1(\widetilde{X},\tilde{x}) \cong p_*\pi_1(\widetilde{X},\tilde{x}) $ is a subgroup of $ \pi_1(X,x) $. A subgroup $ H $ of the fundamental group $ \pi_1(X,x) $ is called covering, semicovering and generalized covering subgroup if there is a covering, semicovering and generalized covering map $ p: (\widetilde{X},\tilde{x}) \rightarrow (X,x)  $ such that $ H = p_*\pi_1(\widetilde{X},\tilde{x}) $, respectively.

In order to classification of various types of covering subgroups in the fundamental group using their topological properties on $ \pi_1^{qtop}(X,x) $, Brazas  gathered some results in a diagram \cite[page 288]{BrazSO}. More precisely, it was shown that for a connected locally path connected space $ X $, a subgroup $ H \leq \pi_1(X,x) $ is a semicovering subgroup if and only if $ H $ is open in $ \pi_1^{qtop}(X,x) $. It seems interesting to express similar results for other types of coverings, using another topologies on the fundamental group. The Spanier subgroup topology is a suitable one to characterize covering subgroups and lead us to a class of topologies on groups which is called subgroup topology. 

In this class of topologies on a group, a collection of subgroups with the finite intersection property, which is called the neighbourhood family, creates a local base for the trivial element. This local base can be transferred to all elements of the group, since the left translation maps are continuous. Therefore, the collection of all left cosets of subgroups contained in the neighbourhood family forms the subgroup topology on the group. Bogley et al. \cite{Bog} introduced two types of the subgroup topologies on the fundamental groupoid and studied the properties of the fibres from the endpoint projection map. In Section \ref{Sec2}, we study some general properties of subgroup topologies and show that a group $ G $ equipped with the subgroup topology is a topological group if and only if all its right translation maps are continuous (Proposition \ref{pr2.1}). Then, by extending the concept of coverable spaces, we introduce different classes of coverability for a variety of coverings using the subgroup topologies on the fundamental groups.

In Section \ref{Sec3}, some types of subgroup topologies on the fundamental group and its properties are studied. As mentioned previously, the Spanier subgroup topology, which determined by the collection of all Spanier subgroups as the neighbourhood family, characterize a well-known classification of covering subgroups such as: \textit{A subgroup $ H $ of the fundamental group is a covering subgroup if and only if $ H $ is an open subgroup of the Spanier subgroup topology (Theorem \ref{th31})}. In order to study different topologies on the fundamental group, we show in Proposition \ref{pr3.4} that the \textit{lasso topology} on the fundamental group, which was introduced in \cite{Brod7},  coincide with the Spanier subgroup topology. Another type of the subgroup topology on the fundamental group is the \textit{path Spanier topology} which its relative neighbourhood family contains all path Spanier subgroups of the fundamental group. In Proposition \ref{pr3.7} it is shown that the discreteness of these two topologies (Spanier and path Spanier subgroup topology) is equivalent to $ X $ be unbased semilocally simply connected. On the other hand, Wilkins \cite{Wil} showed that if all elements of the neighbourhood family of a subgroup topology on a group $ G $ are normal subgroups, then $ G $ is a topological group. Although, an arbitrary path Spanier subgroup of the fundamental group does not necessary be normal, in general, we show that the path Spanier subgroup topology always make the fundamental group a topological group (Proposition \ref{pr3.8}). 

In continue, we compare these topologies with the other known types of topologies on the fundamental group such as the inherited topology from the compact-open topology, which is called the \textit{qtop-topology}, the  \textit{$ \tau $- topology} which was introduced in \cite{BrazT}, the \textit{whisker topology} and the \textit{gcov-topology} (Definition \ref{de321}). Recall from \cite[Lemma 3.1]{paper1} that the whisker topology is another type of the subgroup topology on the fundamental group. Indeed, the ralationship between the mentioned topologies on the fundamental group of locally path connected spaces is gathered in Chain $ (*) $. Some examples and counterexamples show that these topologies may be different, in general. Moreover, the diagram shows the relationship of discreteness of the subgroup topologies together.


\section{Subgroup Topology} \label{Sec2}

The \textit{subgroup topology} on a group $G$ specified by a family of subgroups of $G$ was defined in \cite[section 2.5]{Bog}
 and considered by some recent researchers such as \cite{Wil,BrazFabTH}. The collection $\Sigma$ of subgroups of $G$ is called a
 \textit{neighbourhood family} if for any $H,K \in \Sigma$, there is a subgroup $S \in \Sigma$ such that $S \subseteq H\cap K$. As a result of this property, the collection of all left cosets of elements of $\Sigma$ forms a basis for a topology on $G$, which is called the subgroup topology determined by $\Sigma$. Bogley et al. \cite{Bog} focused on some general properties of subgroup topologies and showed that they are homogeneous spaces, since left translation by elements of $G$ determine self-homemorphisms of $G$. Also,  they introduced the intersection $S_\Sigma=\cap \{H \ \vert \ H \in \Sigma\}$, called \textit{infinitesimal} subgroup for the neighbourhood family $\Sigma$ and showed that the closure of the element $g\in G$ is the coset $g S_\Sigma$. Although it is pointed out in \cite{Bog} that the group $ G $ equipped with a subgroup topology in general may not necessarily a topological group  (it may not even a quasitopological group), because right translation maps by a fixed element of $G$ need not be continuous, but it has some of properties of topological groups  (for more details see Theorem 2.9 from \cite{Bog}).  Moreover, if $ H $ is a subgroup of $ G $, and $ K $ is a subgroup of $ H $ which is open in $ G $ topologized with a subgroup topology, then $ H $ is also open in $ G $ since $ H $ decomposes as a union of open cosets of $ K $.\\

On the other hand, Wilkins \cite[Lemma 5.4]{Wil} showed that a group $ G $ with the subgroup topology determined by a neighbourhood family $ \Sigma $  is a topological group when all subgroups in $\Sigma$ are normal. Since all left translation maps by elements of a group $G$ equipped with a subgroup topology are continuous, then the group $G$ is a left topological group by the sense of Arhangeliskii's topological groups \cite[page 12]{Arx}. In the following proposition we show that if right translation maps by elements of $ G $ are also continuous, then $G$ will be a topological group. \\
Note that a right translation map $r_t: G \to G$ by the element $t \in G$, is $r_t(g)=gt \quad \forall g \in G$.

\begin{proposition} \label{pr2.1}
Let $G$ be a group equipped with the subgroup topology determined by the neighbourhood family $\Sigma$. If all right translation maps are continuous, then $G$ is a topological group.
\end{proposition}

\begin{proof}
It is enough to show continuity of operations taking inverse and multiplication. Let  $f: G \to G$ defined by $g \mapsto g^{-1}$ be the inverse operation and fix $g\in G$. Clearly, for every $ H \in \Sigma $, $ g^{-1}H $ is a basis open neighbourhood of the subgroup topology containing $g^{-1} $ ( Note that for any $ sH $ containing $ g^{-1} $ we have $ sH=g^{-1}H $). By hypothesis, the right translation map $ r_{g^{-1}}:G \rightarrow G $ with $ r_{g^{-1}}(s) = sg^{-1} $ is continuous. Then, for any $ s \in  G $ and for every $ H \in \Sigma $ there is a $ K \in \Sigma $ such that 
\[
 sKg^{-1}=r_{g^{-1}}(sK) \subseteq sg^{-1}H 
\]
 and so $ Kg^{-1} \subseteq g^{-1}H $. Now for such $ K $,
 \[
  f(gK)=Kg^{-1} \subseteq g^{-1}H 
  \]
   which shows that $f$ is continuous. For continuity of the multiplication map $m: G \times G \to G$ defined by $m: (g_1, g_2) \mapsto g_1g_2$, let $g_1g_2H$ be a basis open neighbourhood of $G$ containing $g_1g_2$ for $ H \in \Sigma $. Applying the continuity of taking inverse for the element $g_2^{-1}\in G$,  implies that for every $ H \in \Sigma $ there exists a subgroup $K \in \Sigma$ such that $ K g_2=f(g_2^{-1}K) \subseteq g_2H$. Therefore,
\[
m(g_1K, g_2 H)=g_1 K g_2 H \subseteq g_1g_2H,
\]
which shows that the multiplication map is continuous under product topology.
\end{proof}

Clearly, every topological group is also a left and right topological group. The following corollary is the immediate consequence of this fact and the above proposition. 

\begin{corollary}
A group equipped with a subgroup topology is a topological group if and only if all right translation maps are continuous.
\end{corollary}

Pakdaman et al. \cite[Definition 2.4]{MTPUCo} introduced the notion of coverable spaces in such a way that a pointed topological space $ (X,x_0) $ is called coverable if it has the categorical universal covering space or equivalently the Spanier group, $ \pi_1^{sp}(X,x_0) $, is a covering subgroup. Recall that $ \pi_1^{sp}(X,x_0) $ is the intersection of all of the Spanier subgroups $ \pi(\mathcal{U},x_0) $, where $ \mathcal{U} $ is an open cover of $ X $ i.e. $ \pi_1^{sp}(X,x_0) $ is the infinitesimal subgroup of Spanier subgroup topology on the fundamental group (for more details see the next section). Therefore, a topological space $ X $ is coverable if and only if the infinitesimal subgroup of the Spanier subgroup topology is open.
Note that the infinitesimal subgroup $ S_\Sigma $ of $ G $ need not be an open subgroup, in general. However, some nice properties may occur when $ S_\Sigma $ is open. In the case of  fundamental groups one can guess the following notion (see \cite{MTPUCo}).

 \begin{definition}
Let $ (X,x_0) $ be a pointed topological space and $ \pi_1(X,x_0) $ be the fundamental group equipped with the subgroup topology which determined by the neighbourhood family $\Sigma$. Then $ X $ is called \textit{$ \Sigma$-coverable} if the infinitesimal subgroup $ S_\Sigma $ is open in $\pi_1^{\Sigma}(X,x_0)$.
\end{definition}

Clearly, the infinitesimal subgroup $S_\Sigma$ is open in $G$ if and only if $ S_\Sigma \in \Sigma $. Moreover, if $ S_\Sigma \in \Sigma $, then any intersection of open subgroups of $ G $ are open. Moreover, it can be seen that in every left (right) topological groups, any open subgroup is closed but the converse does not hold, in general. The following proposition shows that it will be hold when the infinitesimal subgroup is an open subgroup.

\begin{proposition} \label{pr2.3}
Let $\pi_1^{\Sigma}(X,x_0)$ be the fundamental group of $ (X,x_0) $ equipped with a subgroup topology determined by $ \Sigma $. Then the following statements are equivalent.

\begin{enumerate}
\item  $ X $ is $ \Sigma$-coverable.
\item  Every closed subgroup of $\pi_1^{\Sigma}(X,x_0)$ is an open subgroup.
\item  A subgroup $ H $ of $\pi_1^{\Sigma}(X,x_0)$ is open if and only if it is closed.
\item  A subgroup $ H $ of $\pi_1^{\Sigma}(X,x_0)$ is open if and only if $ S_\Sigma \leq H $.
\end{enumerate}
\end{proposition}
\begin{proof}
$(1)\leftrightarrow (2)$ Let $ K $ be a closed subgroup of $\pi_{1}^{\Sigma}(X,x_{0})$ and put $ g \in K $. Since $ gS_\Sigma $ is the closure of $ g $, then $ gS_\Sigma \subseteq K$ and hence $ S_\Sigma \subseteq K$. It shows that $ K $ is open. The converse is trivial since $ S_\Sigma $ is a closed subgroup of $ G $.

$(2)\leftrightarrow (3)$ This is an immediate of the fact that $\pi_{1}^{\Sigma}(X,x_{0})$ is a left topological group.

$(1)\leftrightarrow (4)$ By definition if $X$ is $\Sigma$-\textbf{coverable}, then $ S_\Sigma $ is open and thus so is any subgroup $ H $ containing $ S_\Sigma $. The converse follows directly from the definition.
\end{proof}

\begin{remark}
Note that $ G $ equipped with the subgroup topology determined by the neighborhood family $ \Sigma $ is discrete if and only if the trivial subgroup belong to $ \Sigma $ and so $ S_\Sigma=1 $.
\end{remark}

It is well-known that the canonical group homomorphism $ \varphi: \pi_1(X,x_0) \times \pi_1(Y,y_0) \rightarrow \pi_1(X \times Y,(x_0,y_0))$ is an isomorphism. The question now is if the fundamental groups equipped with a topology, does $ \varphi $ become homeomorphism? Clearly, it is done when the fundamental groups are topological  groups with the topology they are equipped with. Brazas and Fabel \cite[Lemma 41]{BrazFab} showed that it does not hold for the induced topology from the compact-open topology where $ \pi_1^{qtop}(X \times Y,(x_0,y_0))$ is not a topological group. In the following we show that it is true for the fundamental groups equipped with a subgroup topology.

\begin{proposition} \label{pr26}
If the fundamental groups of pointed topological spaces $ (X,x_0) $ and $ (Y,y_0) $ equipped with subgroup topologies, then the canonical isomorphism $ \varphi: \pi_1(X,x_0) \times \pi_1(Y,y_0) \rightarrow \pi_1(X \times Y,(x_0,y_0)) $ is a homeomorphism.
\end{proposition}

\begin{proof}
Let $ \Sigma_X $ and $ \Sigma_Y $ be neighbourhood families of $ \pi_1(X,x_0) $ and $ \pi_1(Y,y_0) $, respectively. Put $ \Sigma_{X \times Y}=\{ H \leq \pi_1(X \times Y,(x_0,y_0)) \ | \ H=H_X \times H_Y, \ H_X \in  \Sigma_X, \ H_Y \in \Sigma_Y$. For every pair $ H, K \in \Sigma_{X \times Y} $ since  $ H \cap K = H_X \times H_Y \cap K_X \times K_Y = (H_X \cap K_X) \times (H_Y \cap K_Y) \in \Sigma_X \times   \Sigma_Y = \Sigma_{X \times Y} $, then $ \Sigma_{X \times Y} $ forms a neighbourhood family on $ \pi_1(X \times Y,(x_0,y_0)) $. Now it is enough to show that $ \varphi: \pi_1^{\Sigma_X}(X,x_0) \times \pi_1^{\Sigma_Y}(Y,y_0) \rightarrow \pi_1^{\Sigma_{X \times Y}}(X \times Y,(x_0,y_0))$ and $ \varphi^{-1} $ are continuous. For every $ [\alpha] \in \pi_1^{\Sigma_{X \times Y}}(X \times Y,(x_0,y_0)) $ and $  H \in  \Sigma_{X \times Y} $ by the definition we have $ \varphi([\alpha_X]H_X,[\alpha_Y]H_Y) = [\alpha]H $, where $ \alpha_X  $ and $ \alpha_Y $ are projections of $ \alpha $ in $ X $ and $ Y$, respectively. Thus $ \varphi $ is continuous. Moreover, since for $ H_X \in \Sigma_X$ and $ H_Y \in \Sigma_Y $ with $ H_X \times H_Y = H $, $ \varphi^{-1}(H) = (H_X, H_Y) $, then $ \varphi^{-1} $ also is  continuous.
\end{proof}



\section{Some Subgroup Topologies on the Fundamental Group} \label{Sec3}

For a topological space $X$, the fundamental group $\pi_1(X, x_0)$ admits a variety of distinct natural subgroup topologies \cite{Bog,Wil}, which some of them have been studied to find some properties of $\pi_1(X, x_0)$. As an example, \textit{Spanier subgroup topology} \cite[page 12]{Wil} was introduced using the collection of all Spanier subgroups $\pi(\mathcal{U}, x_0)$ of the fundamental group $\pi_1(X, x_0)$ as the neighbourhood family $ \Sigma^S $. Recall that \cite[Page 81]{Span}, the Spanier subgroup determined by an open covering $\mathcal{U}$ of $X$ is the normal subgroup $\pi(\mathcal{U}, x_0)$ of $\pi_1(X, x_0)$ generated by the homotopy class of lollipops $\alpha * \beta * \alpha^{-1}$, where $\beta$ is a loop lying in an element of $U \in \mathcal{U}$ at $\alpha(1)$, and $\alpha$ is any path originated at $x_0$. The fundamental group equipped with the Spanier subgroup topology is denoted by $\pi_1 ^{\mathrm{Span}}(X, x_0)$.  From Lemma 5.4 of \cite{Wil}, it is clear that $\pi_1 ^{\mathrm{Span}}(X, x_0)$ is a topological group since every $\pi(\mathcal{U}, x_0) $ is a normal subgroup.

The following interesting  classical result of Spanier \cite[Section 2.5 Theorems 12,13]{Span} realized the relationship between classical covering space theory and the Spanier subgroups of the fundamental group.

\begin{theorem} \label{th31}
Let $X$ be a connected locally path connected space and $H \leq \pi_1(X, x_0)$. Then there exists  a covering projection $p:\widetilde{X}\to X$ with $p_* (\pi_1 (\widetilde{X},\tilde{x}))=H$ (or equivalently, $ H $ is a covering subgroup of $ \pi_1(X, x_0) $) if and only if there exists an open cover $\mathcal{U}$ of $X$ in which $\pi(\mathcal{U}, x_0)\leq H$, or equivalently, $ H $ is an open subgroup of $ \pi_1 ^{\mathrm{Span}}(X, x_0) $.
\end{theorem}

 \begin{remark}
 Fischer et al. \cite{FiZa4} distinguished the notions based and unbased semilocally simply connectedness and showed that pointed topological space $(X,x_{0}) $  is unbased semilocally simply connected if and only if there exists an open covering $\mathcal{U}$ of $X$ such that $\pi(\mathcal{U}, x_0)$ is trivial. This statement can be recreated as follows.
 \end{remark}
 
 \begin{proposition} \label{pr33}
  A pointed topological space $ (X,x_0) $  is unbased semilocally simply connected if and only if $\pi_1^{\mathrm{Span}}(X, x_0)$ is discrete.
\end{proposition}

\begin{proof}
If $\pi_1^{\mathrm{Span}}(X, x_0)$ is discrete, then the trivial subgroup is open in $\pi_1^{\mathrm{Span}}(X, x_0)$ and so there is an open cover $ \mathcal{U} $ of $ X $ such that $ \pi(\mathcal{U}, x_0) = 1 $, i.e, $ X $ is unbased semilocally simply connected. Conversely, if $ X $ has an open cover $ \mathcal{U} $ with $ \pi(\mathcal{U}, x_0) = 1 $, then $ \{1\} $ is open. Since $\pi_1^{\mathrm{Span}}(X, x_0)$ is a topological group, then by using of translation maps, every  $ [\alpha] \in \pi_1^{\mathrm{Span}}(X, x_0)  $ is open in $\pi_1^{\mathrm{Span}}(X, x_0)$. Therefore, $\pi_1^{\mathrm{Span}}(X, x_0)$ is discrete.
\end{proof}

Pakdaman et al. \cite{MTPUCo} introduced the concepts of coverable and semilocally Spanier spaces and showed that these notions are equivalent in the case of connected locally path connected spaces \cite[Theorem 2.8]{MTPUCo}. Note that the infinitesimal subgroup of $\pi_1^{\mathrm{Span}}(X, x_0)$ named  the Spanier group and denoted by $\pi_1^{sp}(X, x_0)$ \cite{FiZa4}. The following proposition adds another equivalent to them.

\begin{proposition} \label{pr3.3}
For a connected and locally path connected space X, the following statements are equivalent.

\begin{enumerate}
\item  $ X $ is a $ \Sigma^S $-coverable space (or coverable in the sense of \cite{MTPUCo}).
\item  $ X $ is a semilocally Spanier space.
\item  $\pi_1^{sp}(X, x_0)$ is an open subgroup of $ \pi_1^{\mathrm{Span}}(X, x_0) $.
\end{enumerate}

\end{proposition}

On the other hand, Brodskiy et al. \cite[Section 3]{Brod7}  introduced another topology on the universal path space $\widetilde{X}$ using open coverings of $X$, which makes the fundamental group a topological group \cite[Proposition 5.17]{Brod7} and named it \textit{lasso topology}.

Recall from \cite[definition 4.11]{Brod} that for any topological space $X$, the lasso topology on the set $\widetilde{X}$ is defined by the basis $  N(\langle \alpha \rangle , \mathcal{U}, W)$, where $ \alpha $ is a path originated at $ x_0 $,  $W$ is a neighbourhood of the endpoint $\alpha(1)$ and $\mathcal{U}$ is an open cover of $X$. A class $\langle \gamma \rangle \in \widetilde{X}$ belongs  to $ N(\langle \alpha \rangle , \mathcal{U}, W)$ if and only if it has a representation of the form $\alpha * L * \beta$ where $[L]$ belongs to $\pi\big(\mathcal{U},\alpha(1)\big)$ and $\beta$ is a based loop in $W$ at $\alpha(1)$.

There is a bijection between the fundamental group $\pi_1(X, x_0)$ and the fibre of the base point $p^{-1}(x_0)$, where $p: \widetilde{X}\to X$ is the endpoint projection map. Therefore, the fundamental group $\pi_1(X, x_0)$ as a subspace of the universal path space $\widetilde{X}$ inherits any topology from $\widetilde{X}$. Thus, the collection of sets with the form $ N(\langle \alpha\rangle, \mathcal{U}, W)\cap p^{-1}(x_0)$ is a basis for the lasso topology on $\pi_1(X, x_0)$, which we denote it by $\pi_1^{\mathrm{lasso}} (X, x_0)$.

Brodskiy et al. \cite[Section 3]{Brod7} stated some properties of $\pi_1^{\mathrm{lasso}}(X, x_0)$ and relationships between covering subgroups and the lasso topology on the fundamental group. In the following we show that the lasso topology on the fundamental group and the Spanier subgroup topology coincide, in general.

\begin{proposition} \label{pr3.4}
Let $X$ be a topological space. The lasso topology on the fundamental group $\pi_1(X, x_0)$ coincides with the Spanier subgroup topology.
\end{proposition}

\begin{proof}
Let $\beta_1$ be the basis of lasso topology on the fundamental group consists of the sets of the form $ N(\langle \alpha\rangle, \mathcal{U}, W)\cap p^{-1}(x_0)$. Since for every $[\alpha]\in \pi_1(X, x_0)$ and any open cover $\mathcal{U}$ of $X$, the set $[\alpha]\pi (\mathcal{U}, x_0)$ belongs to $\beta_1$, then the lasso topology on the fundamental group is finer than the Spanier subgroup topology. \\
Conversely, let $S$ be an open subset of $\pi_1^{\mathrm{lasso}}(X, x_0)$ and $[\alpha]\in S \subseteq \pi_1 (X, x_0)$. Then there exists an open basis neighborhood $ N(\langle \mu\rangle , \mathcal{U}, W)\cap p^{-1}(x_0)$ of $[\alpha]$ which contained in $S$. We show that $[\alpha] \pi (\mathcal{U}, x_0) \subseteq S$.
Since $[\alpha] \in N(\langle \mu \rangle, \mathcal{U}, W)\cap p^{-1}(x_0)$, then there exists $[\eta] \in \pi \big(\mathcal{U}, \mu(1)\big)$ and $\lambda: I \rightarrow W$ with $\lambda (0)=\mu(1)$ and $\lambda (1)=x_0$, such that $\alpha \simeq \mu *\eta * \lambda$. Now for any $[\xi] \in \pi(\mathcal{U}, x_0)$ with $\xi \simeq \prod^n_{i=1} \delta_i * \gamma_i * \delta _i^{-1}$ we have:
\begin{align*}
\alpha * \xi & \simeq \mu * \eta *\lambda *\delta_1 *\gamma_1 * \delta_1 ^{-1} * \delta _2 *\gamma_2* \delta_2^{-1}* \cdots * \delta _n* \gamma_n* \delta _n^{-1}\\
&\simeq \mu * \eta *\lambda *\delta_1 *\gamma_1 *\delta_1^{-1} * \lambda ^{-1} * \lambda *\delta_2 *\gamma_2 *\delta_2^{-1} *\lambda ^{-1} \\
&\quad * \cdots * \lambda * \delta_n * \gamma_n * \delta_n^{-1} * \lambda ^{-1} * \lambda.
\end{align*}
Put $ \varrho \simeq \prod^n_{i=1} \lambda * \delta_i * \gamma_i* \delta_i^{-1} * \lambda ^{-1}\in \pi \big(\mathcal{U}, \mu(1)\big)$, then
\[
\alpha* \xi \simeq \mu * \eta * \varrho *\lambda \in N(\langle \mu \rangle, \mathcal{U}, W)\cap p^{-1}(x_0) \subseteq S.
\]
Therefore $[\alpha]\pi (\mathcal{U},x_0)\subseteq N(\langle \mu\rangle, \mathcal{U}, W)\cap p^{-1}(x_0) \subseteq S$.
\end{proof}

Torabi et al. \cite[Section 3]{MTPSp} replaced open covers with path open covers of the space $ X $ in the definition of Spanier subgroups and introduced path Spanier subgroups by the same way. Recall that a path open cover $ \mathcal{V} $ of the path component of $ X $ involve $ x_0 $ is the collection of open subsets $\lbrace V_{\alpha} \ \vert \ \alpha\in{P(X,x_{0})}\rbrace$ and the path Spanier subgroup $\widetilde{\pi}(\mathcal{V}, x_{0})$ with respect to the path open cover $ \mathcal{V} $ is the subgroup of $ \pi_1(X, x_0)$ consists of all homotopy classes having representatives of the following type:\\
$$\prod_{j=1}^{n}\alpha_{j}\beta_{j}\alpha^{-1}_{j},$$
where $\alpha_{j}$'s are arbitrary path starting at $x_{0}$ and each $\beta_{j}$ is a loop inside of the open set $V_{\alpha_{j}}$ for all $j\in{\lbrace1,2,...,n\rbrace}$.

If $ \mathcal{U} $ and $ \mathcal{V} $ are two path open covers of a space $ X $, the collection $ \mathcal{W}=\lbrace U_{\alpha} \cap V_{\alpha} \ \vert \ \forall \alpha\in{P(X,x_{0}), U_{\alpha} \in \mathcal{U} \ and \ V_{\alpha} \in \mathcal{V} }\rbrace $ is a refinement of both $ \mathcal{U} $ and $ \mathcal{V} $. Thus, $\widetilde{\pi}(\mathcal{W}, x_{0}) \leq  \widetilde{\pi}(\mathcal{U}, x_{0}) \cap \widetilde{\pi}(\mathcal{V}, x_{0}) $, which shows that the collection of all path Spanier subgroups of the fundamental group forms a neighbourhood family.

\begin{definition} \label{def35}
 For a pointed space $ (X,x_0) $, let $ \Sigma^{\mathrm{P}} $ be the collection of all path Spanier subgroups of  $ \pi_1(X, x_0) $. We call the subgroup topology determined by $\Sigma^{\mathrm{P}} $ the \textit{path Spanier topology} and denote it by $ \pi_1^{\mathrm{pSpan}}(X, x_0)$.
 \end{definition}

Moreover, they showed that if a path Spanier subgroup $ \widetilde{\pi}(\mathcal{V}, x_{0})  $ is normal, then there exists an Spanier subgroup $ {\pi}(\mathcal{U}, x_{0})  $ for which $ \widetilde{\pi}(\mathcal{V}, x_{0}) = {\pi}(\mathcal{U}, x_{0}) $ \cite[Theorem 3.2]{MTPSp}. The following proposition appears as a result of this fact.

 \begin{proposition} \label{pr3.7}
For a locally path connected space $X$, $\pi_1^{\mathrm{pSpan}}(X, x_0)$ is discrete if and only if $\pi_1^{\mathrm{Span}}(X, x_0)$ is discrete.
\end{proposition}

 \begin{proof}
 By definition every open cover is also a path open cover, hence $\pi_1^{\mathrm{pSpan}}(X, x_0)$ is finer than $\pi_1^{\mathrm{Span}}(X, x_0)$ for any space $ X $. Therefore, $\pi_1^{\mathrm{pSpan}}(X, x_0)$ is discrete when $\pi_1^{\mathrm{Span}}(X, x_0)$ be discrete. Conversely, if $\pi_1^{\mathrm{pSpan}}(X, x_0)$ is discrete, then there is a trivial path Spanier subgroup, i.e, $ \widetilde{\pi}(\mathcal{V}, x_{0}) =1  $ and so $ \widetilde{\pi}(\mathcal{V}, x_{0})  $ is a normal subgroup. Now one can conclude from \cite[Theorem 3.2]{MTPSp} that there exists an open cover $ \mathcal{U} $ of $ X $ such that  ${\pi}(\mathcal{U}, x_{0}) = \widetilde{\pi}(\mathcal{V}, x_{0}) = 1 $. Therefore, $\pi_1^{\mathrm{Span}}(X, x_0)$ is also discrete.
 
\end{proof}

The following corollary is obtained from the combination of the above proposition and Proposition \ref{pr33}.

\begin{corollary} \label{co38}
For a locally path connected space $X$, the following statements are equivalent.
\begin{enumerate}
\item
$ X $ is unbased semilocally simply connected space.
\item
$\pi_1^{\mathrm{Span}}(X, x_0)$ is discrete.
\item
$\pi_1^{\mathrm{pSpan}}(X, x_0)$ is discrete.
\end{enumerate}

Moreover, each of the above statements implies that

\item $\pi_1^{\mathrm{Span}}(X, x_0) = \pi_1^{\mathrm{pSpan}}(X, x_0)$.

and
\item
$\widetilde\pi_1^{sp} (X,x_0) = \pi_1^{sp} (X,x_0) $.

\end{corollary}

Brazas \cite[Theorem 5.5]{BrazS} showed that for a locally path connected space $ X $, the map $ p: X \rightarrow Y $ is a semicovering map if and only if the image of the relative induced homomorphism $ p_*:  \pi_1(\widetilde{X}, \tilde{x}_0) \rightarrow \pi_1(X, x_0)$ is an open subgroup of $ \pi_1^{\mathrm{qtop}}(X, x_0) $, where $\pi_1^{\mathrm{qtop}}(X, x_0)$ is the fundamental group equipped with the compact-open topology inherited from the loop space by quotient map.
On the other hand, Torabi et al. \cite[Theorem 3.3]{MTPSp} stated that for a locally path connected space $ X $ every path Spanier subgroups are open in $\pi_1^{\mathrm{qtop}}(X, x_0)$.  Moreover, they showed that \cite[Corollary 3.4]{MTPSp} a subgroup $ H $ of $\pi_1^{\mathrm{qtop}}(X, x_0)$  is open if and only if there exists a path open cover $ \mathcal{V} $ of $ X $ such that $\widetilde{\pi}(\mathcal{V}, x_{0}) \leq H$. Finally, they concluded  the relationship between semicovering subgroups and path Spanier subgroups in the fundamental group as follows.

\begin{lemma} \cite[Theorem 4.1]{MTPSp} 
Let $X$ be a connected locally path connected space. A subgroup $ H $ of $\pi_1(X, x_0)$ is a semicovering subgroup if and only if there is a path open cover $ \mathcal{V} $ of $ X $ such that $\widetilde{\pi}(\mathcal{V}, x_{0}) \leq H$.
\end{lemma}

The following proposition is the immediate consequence of Definition \ref{def35} and the above lemma.

 \begin{proposition}
 For a locally path connected space $ X $ a subgroup $ H $ of the fundamental group $\pi_1(X, x_0)$ is a semicovering subgroup if and only if $ H $ is open in $\pi_1^{\mathrm{pSpan}}(X, x_0)$.
 \end{proposition}

Wilkins \cite[Lemma 5.4]{Wil} showed that a group $ G $ with the subgroup topology determined by a neighbourhood family $ \Sigma $  is a topological group when all subgroups in $\Sigma$ are normal.
Although for an arbitrary path open cover $ \mathcal{V} $ of $ X $, the path Spanier subgroup $ \widetilde{\pi}(\mathcal{V}, x_{0})  $ may not be a normal subgroup, in general \cite[Theorem 3.2]{MTPSp}, the following proposition shows that $\pi_1^{\mathrm{pSpan}}(X, x_0)$ is a topological group.

\begin{proposition} \label{pr3.8}
For any space $X$, the fundamental group $\pi_1^{\mathrm{pSpan}}(X, x_0)$ equipped with the path Spanier topology is a topological group.
\end{proposition}

\begin{proof}
Since $\pi_1^{\mathrm{pSpan}}(X, x_0)$ is a subgroup topology, by Proposition \ref{pr2.1} it is enough to show that right translation maps are continuous. Let $ [\alpha] \in \pi_1(X, x_0) $ and $ r_{\alpha}: \pi_1^{\mathrm{pSpan}}(X, x_0) \rightarrow \pi_1^{\mathrm{pSpan}}(X, x_0) $ with $ r_{\alpha}([\beta] ) = [\beta*\alpha] $ for any $ [\beta] \in \pi_1(X, x_0) $ be the right translation map with respect to $[\alpha]$. If $  \mathcal{V} $ is an arbitrary path open cover of $ X $, then $ [\beta*\alpha]\widetilde{\pi}(\mathcal{V}, x_{0})  $ is a basis open neighbourhood of $\pi_1^{\mathrm{pSpan}}(X, x_0)$ at $ [\beta*\alpha] $. By definition, for any path $ \gamma \in P(X,x_0) $ there is a $ V_{\gamma} \in \mathcal{V}  $ with $ \gamma(1) \in V_{\gamma} $. Put $ W_{\gamma}=V_{\gamma} \cap V_{\alpha^{-1}*\gamma} $, then the collection $ \mathcal{W}= \lbrace W_{\gamma} \ \vert \ \gamma \in P(X,x_{0}) \rbrace $ is also a path open cover of $ X $, which is a refinement of $ \mathcal{V} $. Thus, as an immediate consequence of the definition of path Spanier subgroups, we have: \\
$$   \widetilde{\pi}(\mathcal{W}, x_{0}) \subseteq \widetilde{\pi}(\mathcal{V}, x_{0}).   \ \ \ \ \ \ \ (*) $$

Let $[\prod_{j=1}^{n}\gamma_{j}*\delta_{j}*\gamma^{-1}_{j}] \in \widetilde{\pi}(\mathcal{W}, x_{0})$ be an arbitrary homotopy class of a product of lollipops in $ \widetilde{\pi}(\mathcal{W}, x_{0}) $. For the map $ r_{\alpha} $ we have:
$$  r_{\alpha}([\beta*\prod_{j=1}^{n}\gamma_{j}*\delta_{j}*\gamma^{-1}_{j}])$$
 $$= r_{\alpha}([\beta*\alpha*\prod_{j=1}^{n}(\alpha^{-1}*\gamma_{j}*\delta_{j}*\gamma^{-1}_{j}*\alpha)*\alpha^{-1}])$$
  $$ = [\beta*\alpha*\prod_{j=1}^{n}(\alpha^{-1}*\gamma_{j}*\delta_{j}*\gamma^{-1}_{j}*\alpha)*\alpha^{-1}* \alpha ]$$
  $$ = [\beta*\alpha*\prod_{j=1}^{n}(\alpha^{-1}*\gamma_{j}*\delta_{j}*\gamma^{-1}_{j}*\alpha)] \in [\beta*\alpha] \widetilde{\pi}(\mathcal{W}, x_{0}). $$
  Therefore, from $ (*) $ we have
  $$ r_{\alpha}([\beta] \widetilde{\pi}(\mathcal{W}, x_{0})) \subseteq [\beta*\alpha] \widetilde{\pi}(\mathcal{V}, x_{0}), $$ which shows that $ r_{\alpha} $ is a continuous map.
\end{proof}

Note that, for a locally path connected space $ X $, open subgroups of $\pi_1^{\mathrm{pSpan}}(X, x_0)$ and $\pi_1^{\mathrm{qtop}}(X, x_0)$ coincide, but it may not hold, in general. As an example, consider Figure 1 of \cite{VirZaCam} for which $\pi_1^{\mathrm{qtop}}(X, x_0)$ is not discrete and so the trivial subgroup is not open in $\pi_1^{\mathrm{qtop}}(X, x_0)$. On the other hand, since the space is semilocally simply connected, $\pi_1^{\mathrm{pSpan}}(X, x_0)$ is discrete by Proposition \ref{pr3.7}. Then the trivial subgroup is open in $\pi_1^{\mathrm{pSpan}}(X, x_0)$.

For a locally path connected space $ X $, let $ B $ be an open subset of $\pi_1^{\mathrm{pSpan}}(X, x_0)$ and $ [\beta]\in B $. From the definition, there exists a path open cover $ \mathcal{V} $ of $ X $  such that $ [\beta]\widetilde{\pi}(\mathcal{V}, x_{0}) \subseteq B $. Recall from \cite[Theorem 3.3]{MTPSp} that $ \widetilde{\pi}(\mathcal{V}, x_{0})  $ and then $ [\beta]\widetilde{\pi}(\mathcal{V}, x_{0})  $ are open subsets of $\pi_1^{\mathrm{qtop}}(X, x_0)$. Hence $  B $ is an open subset of $\pi_1^{\mathrm{qtop}}(X, x_0)$. Therefore, $\pi_1^{\mathrm{qtop}}(X, x_0)$ is finer than $\pi_1^{\mathrm{pSpan}}(X, x_0)$.
Clearly, this result holds for any topology on the fundamental group which makes it a left topological group and its open subgroups are coincide with open subgroups of $\pi_1^{\mathrm{qtop}}(X, x_0)$.

\begin{proposition}  \label{pr310}
Let $ (X,x_0) $ be a locally path connected space. If $ \pi_1^{\mathrm{*}}(X, x_0) $ is a left topological group on the fundamental group in which its open subgroups are coincide with open subgroups of $\pi_1^{\mathrm{qtop}}(X, x_0)$, then $\pi_1^{\mathrm{*}}(X, x_0)$ is finer than $ \pi_1^{\mathrm{pSpan}}(X, x_0) $.
\end{proposition}

Recall that Brazas introduced in \cite{BrazT} the finest topology on $\pi_{1}(X,x_{0})$ such that $\pi: \Omega(X,x_{0})\rightarrow \pi_{1}(X,x_{0})$ is continuous and $\pi_{1}(X,x_{0})$ is a topological group. The fundamental group with this topology is denoted by $\pi_1^{\mathrm{\tau}}(X, x_0)$.
Also he showed \cite[Proposition 3.16]{BrazT} that for any space $ X $, $\pi_1^{\mathrm{\tau}}(X, x_0)$ and $\pi_1^{\mathrm{qtop}}(X, x_0)$ have the same open subgroups. The following corollary is an immediate consequence of the above proposition and Proposition 3.16 from \cite{BrazT}.

\begin{corollary}
If $X$ is a locally path connected space, then $\pi_1^{\mathrm{pSpan}}(X, x_0)$ is coarser than $\pi_1^{\mathrm{\tau}}(X, x_0)$.
\end{corollary}

 It seems interesting to find the spaces in which $ qtop$-topology and path Spanier topology coincide on the fundamental group. In such spaces, the $ qtop$-topology can be interpreted as a subgroup topology. The following Theorem introduce a class of this spaces.

\begin{theorem}
Let $ X $ be a locally path connected and semilocally small generated space, then $\pi_1^{\mathrm{qtop}}(X, x_0) = \pi_1^{\mathrm{pSpan}}(X, x_0)$.
\end{theorem}

\begin{proof}
Let $ U $ be an arbitrary open subset of $ \pi_1^{\mathrm{qtop}}(X, x_0) $ and take $ [g] \in U $. Clearly, the trivial element of $ \pi_1(X, x_0) $, $ [c_{x_0}] $, belongs to the coset $ [g^{-1}]U $. Since $\pi_1^{\mathrm{qtop}}(X, x_0)$ is a quasitopological group, $ [g^{-1}]U $ is an open subset of $\pi_1^{\mathrm{qtop}}(X, x_0)$. It implies from \cite[Theorem 2.2]{MTPSG} that $ \pi_1^{sg}(X, x_0) \subseteq [g^{-1}]U $.
On the other hand, since $ X $ is a semilocally small generated space, then Theorem 3.8 from \cite{MTPSG} states that $ \pi_1^{sg}(X, x_0) $ is an open subgroup of $\pi_1^{\mathrm{qtop}}(X, x_0)$ and hence it is open in $ \pi_1^{\mathrm{pSpan}}(X, x_0) $. Now $[g] \pi_1^{sg}(X, x_0) \subseteq U $ shows that $ U $ is an open subset of $ \pi_1^{\mathrm{pSpan}}(X, x_0) $ and so $ \pi_1^{\mathrm{pSpan}}(X, x_0) $ is finer than $ \pi_1^{\mathrm{qtop}}(X, x_0) $. The converse statement is easily concluded from Proposition \ref{pr310}.
\end{proof}

 The infinitesimal subgroup of the path Spanier subgroup topology is denoted by $ {\widetilde\pi}^{sp}_1(X,x_0)$. It implies from \cite[Theorem 3.2]{MTPSp} that if $ {\widetilde\pi}^{sp}_1(X,x_0)$ is normal, then  ${\widetilde\pi}^{sp}_1(X,x_0)={\pi}^{sp}_1(X,x_0)$.

Recall from \cite[Definition 4.1]{paper1}  that a space $ X $ is called \textit{semilocally path $ H $-connected} for a subgroup $ H \leq \pi_1(X,x_0) $ if for every path $ \alpha $ beginning at $ x_0 $ there exists an open neighbourhood $ U_\alpha $ of $ \alpha(1) $ with $ i_* \pi_1(U_\alpha ,\alpha(1) ) \leq [\alpha^{-1} H \alpha]$, where $ [\alpha^{-1} H \alpha]= \lbrace [\alpha^{-1}\gamma\alpha] \ \vert \ [\gamma] \in H \rbrace $. The following proposition proposes the same result as Proposition \ref{pr3.3} for the path Spanier topology.

\begin{proposition} \label{pr}
For a connected and locally path connected space X, the following statements are equivalent.
\begin{enumerate}
\item  $ X $ is a $ \Sigma^P $-coverable space.
\item  $ X $ is a semilocally path $ \widetilde{\pi}_1^{sp}(X, x_0) $-connected space.
\item  $\widetilde{\pi}_1^{sp}(X, x_0)$ is an open subgroup of $ \pi_1^{\mathrm{pSpan}}(X, x_0) $.
\end{enumerate}
\end{proposition}

Since every Spanier subgroup of the fundamental group $\pi_1(X, x_0)$ is also a path Spanier subgroup, then for any pointed space $ (X, x_0) $ the path Spanier topology on the fundamental group, $\pi_1^{\mathrm{pSpan}}(X, x_0)$, is finer than the Spanier topology, $\pi_1^{\mathrm{Span}}(X, x_0)$. The following example shows that the converse does not hold, in general.

\begin{example} \label{ex314}
 Recall from  \cite{Hat} and \cite[Remark 3.4]{FiZaCo} that The Hawaiian earring, $ HE $, has a semicovering space which is not a covering space. Therefore, there is a path Spanier subgroup of $\pi_1(X, x_0)$ which is not a Spanier subgroup. This fact implies that $\pi_1^{\mathrm{Span}}(HE,0)$ is not equal to $\pi_1^{\mathrm{pSpan}}(HE,0)$ and hence $\pi_1^{\mathrm{Span}}(HE,0)$ is strictly coarser than $\pi_1^{\mathrm{pSpan}}(HE,0)$. On the other hand, since $ \pi_1^{\mathrm{qtop}}(HE,0) $ is not a topological group, then Proposition \ref{pr3.8} shows that $ \pi_1^{\mathrm{pSpan}}(HE,0) $ is strictly coarser than $ \pi_1^{\mathrm{qtop}}(HE,0) $.
\end{example}

Spanier \cite[page 82]{Span} introduced another topology on the universal path space $ \widetilde{X} $ which has been called the whisker topology by Brodskiy et al. \cite{Brod} and denoted by $ \widetilde{X}^{wh} $. Note that the fundamental group $\pi_{1}(X,x_{0})$ as a subspace of $ \widetilde{X}^{wh} $ inherits the whisker topology which is denoted by $\pi_{1}^{\mathrm{wh}}(X,x_{0})$. Similar to the proof of Proposition \ref{pr3.4}, it is shown in \cite[Lemma 3.1]{paper1} that the collection of the following subsets form a basis for the whisker topology on the fundamental group
$$ \lbrace [\alpha]i_{*}\pi_1(U,x_0)  \ \vert \ [\alpha] \in \pi_1(X,x_{0}) \ \ \& \ \ U  \ \ is \ \ an \ \ open \ \ neighborhood \ \ of \ \ X \ \ at \ \ x_0 \rbrace. $$
It implies that the whisker topology is another type of subgroup topology on the fundamental group determined by the following neighbourhood family of subgroups
$$ \Sigma^{wh}= \lbrace i_{*}\pi_1(U,x_0) \ \vert \  U  \ \ is \ \ an \ \ open \ \ neighborhood \ \ of \ \ X \ \ at \ \ x_0 \rbrace. $$

\begin{remark} \label{re315}
Remember that $\pi_1^{\mathrm{wh}}(HE,0)$ is not a topological group because its right translation maps are not continuous. There is an equivalent condition on a pointed topological space $ (X,x_0) $ which guarantees $\pi_{1}^{\mathrm{wh}}(X,x_{0})$ to be a topological group. Indeed, Jamali et al. \cite[Proposition 2.6]{paper3} proved that $\pi_{1}^{\mathrm{wh}}(X,x_{0})$ is a topological group if and only if $ X  $ is $ SLTL $ at $ x_0 $. The space $ (X,x_0) $ is called $ SLTL $ at $ x_0 $ if for every loop $ \alpha \in \Omega(X,x_0) $ and every open neighborhood $ U $ from $ X $ at $ x_0 $, there exists an open neighborhood $ V $ from $ X $ at $ x_0 $ such that for any loop $ \gamma : (I,\dot{I}) \rightarrow (V,x_0) $, there is a loop $ \lambda : (I,\dot{I}) \rightarrow (U,x_0) $ such that $[\lambda]=[\alpha*\gamma*\alpha^{-1}]$.
Moreover, Brodskiy et al. \cite[Proposition 4.21]{Brod} showed that $\pi_{1}^{\mathrm{wh}}(X,x_{0})$ is discrete if and only if $ X $ is semilocally simply connected at $ x_0 $.
\end{remark}

Fischer and Zastrow \cite[Lemma 2.1]{FiZa} showed that the whisker topology is finer than the $ qtop $-topology on the universal path space $ \widetilde{X} $ for any space $ X $. Clearly, the result will hold for the fundamental group $\pi_{1}(X,x_{0})$ as a subspace of $ \widetilde{X} $.

It implies from \cite[Proposition 3.8]{paper1} that the infinitesimal subgroup of $\pi_{1}^{\mathrm{wh}}(X,x_{0})$ is $\pi_{1}^{{s}}(X,x_{0})$, the collection of all small loops at $ x_0 $. Recall from \cite[Definition 4.1]{paper1} that a topological space $ X $ is called \textit{semilocally $ H $-connected at $ x_0 $} if there is an open neighbourhood $ U $ in $ X $ at $ x_0 $ such that $ i_{*}\pi_1(U,x_0)  \leq H $, for a subgroup $ H  $ of the fundamental group. Note that a topological space $ X $ is called \textit{semilocally simply connected at $ x_0 $} if there is an open neighbourhood $ U $ in $ X $ at $ x_0 $ such that $ i_{*}\pi_1(U,x_0)  = 1 $. The following proposition expresses the relationship between these concepts.

\begin{proposition} \label{pr317}
For any space X, the following statements are equivalent.
\begin{enumerate}
\item  $ X$ is  $ \Sigma^{wh} $-coverable space.
\item  $ X $ is  semilocally  $ \pi_1^{s}(X, x_0) $-connected at $ x_0 $.
\item  $ \pi_1^{s}(X, x_0) $ is an open subgroup of $ \pi_1^{\mathrm{wh}}(X, x_0) $.
\end{enumerate}
\end{proposition}

The following proposition is already expressed and proven by Brodesky et al..

\begin{lemma} \cite[proposition 4.21]{Brod} \label{pr318}
A pointed topological space  $ (X,x_0) $ is semilocally simply connected at $ x_0 $ if and only if $ \pi_1^{\mathrm{wh}}(X, x_0) $ is discrete. 
\end{lemma}

By the above statements, one can summarize the relationship between the mentioned topologies on the fundamental group of locally path connected space $ X $ as the following (Note that we use the symbol $ \subseteq $ to show the finer topology on a group. For example, $ G^{\tau_1} \subseteq G^{\tau_2} $ means that $ \tau_2 $ is finer than $ \tau_1 $ and $ G^{\tau_1} \subsetneq G^{\tau_2} $ means that $ \tau_2 $ is strictly finer than $ \tau_1 $).

$$\pi_1^{\mathrm{Span}}(X, x_0) \subseteq \pi_1^{\mathrm{pSpan}}(X, x_0) \subseteq \pi_1^{\mathrm{\tau}}(X, x_0) \subseteq \pi_1^{\mathrm{qtop}}(X, x_0)  \subseteq \pi_1^{\mathrm{wh}}(X, x_0). \hspace{0.5cm}(*)$$

Using Corollary 3.3 of \cite{paper5} one can introduce the equivalent condition to coincide these topologies on the fundamental group. Recall from \cite[definition 1.3]{paper5} that a pointed topological space $ (X,x_0) $  is called strong small loop transfer (strong $ SLT $ for short) space at $ x_0 $ if for every $ x \in X $ and for every open neighborhood $  U $ of $ X $ containing $ x_0 $ there is an open neighborhood $ V $ containing $ x $ such that for every loop $ \beta: (I,\dot{I}) \rightarrow (V,x) $ and for every path $ \alpha: I \rightarrow X $ from $ x_0 $ to $ x $ there is a loop $ \lambda: (I,\dot{I}) \rightarrow (U,x_0) $ such that $ [\alpha*\beta*\alpha^{-1}] = [\lambda] $. 

\begin{proposition} 
If $ X $ is a path connected space, then $ \pi_1^{\mathrm{Span}}(X, x_0) = \pi_1^{\mathrm{wh}}(X, x_0)  $ if and only if $ X $ is strong $ SLT $ at $ x_0 $  space.
\end{proposition}

\begin{proof}
The result comes from the combination of Corollary 3.3 from \cite{paper5} and Proposition \ref{pr3.4}.
\end{proof}

Brazas in \cite{BrazG} introduced generalized covering spaces inspired by the initial approach of Fischer and Zastrow in \cite{FiZa}. He also introduced generalized covering subgroups of the fundamental group $\pi_1(X, x_0)$ and showed that the intersection of any collection of generalized covering subgroups is also a generalized covering subgroup \cite[Theorem 2.36]{BrazG}. Abdullahi et al. \cite[Lemma 2.10]{paper1} showed that a subgroup $ H $ of the fundamental group $\pi_1(X, x_0)$ is a generalized covering subgroup if and only if $ (p_H)_* \pi_1(\widetilde{X}, \tilde{x}_0) = H $. We intend to introduce another subgroup topology on the fundamental group based on its generalized covering subgroups.

\begin{definition} \label{de321}
 For a pointed space $ (X,x_0) $, let $\Sigma^{\mathrm{g}} $ be the collection of all subgroups $ H $ of  $\pi_1(X, x_0)$ with the property $ (p_H)_* \pi_1(\widetilde{X}, \widetilde{x}_0) = H $. We call the subgroup topology determined by $\Sigma^{\mathrm{g}} $ the \textit{generalized covering topology} and denote it by $\pi_1^{\mathrm{gcov}}(X, x_0)$.
\end{definition}

Abdullahi et al. \cite[Definition 2.3]{paper1} considered the infinitesimal subgroup of the generalized covering topology and denoted it by $ \pi_1^{gc}(X, x_0) $. Also, it was remarked  that $ \pi_1^{gc}(X, x_0) $ is always a generalized covering subgroup and so it is an open subgroup of $\pi_1^{\mathrm{gcov}}(X, x_0)$.  This result implies that any space $ X $ is a $ \Sigma^{\mathrm{g}} $-coverable space.

 As mentioned in the above, for a locally path connected space $ X $, every path Spanier group $ \widetilde{\pi}(\mathcal{V}, x_{0}) $ is an open subgroup of $\pi_1^{\mathrm{qtop}}(X, x_0)$. It is also a closed subgroup since $\pi_1^{\mathrm{qtop}}(X, x_0)$ is a quasitopological group. Recall from \cite[Theorem 2.36]{BrazG} that every closed subgroup of $\pi_1^{\mathrm{qtop}}(X, x_0)$ is a generalized covering subgroup. Then, $ \widetilde{\pi}(\mathcal{V}, x_{0})  \in \Sigma^{\mathrm{g}} $. Therefore, $ \widetilde{\pi}(\mathcal{V}, x_{0}) $ is an open subgroup of $\pi_1^{\mathrm{gcov}}(X, x_0)$. It implies that $\pi_1^{\mathrm{gcov}}(X, x_0)$ is finer than $\pi_1^{\mathrm{pSpan}}(X, x_0)$ in the case of locally path connected spaces. A similar result holds for $ qtop $-topology in the following theorem.
 
\begin{proposition}
  For a connected, locally path connected space $ (X,x_0) $, the generalized covering topology on the fundamental group, $\pi_1^{\mathrm{gcov}}(X, x_0)$, is finer than $\pi_1^{\mathrm{qtop}}(X, x_0)$.
\end{proposition}
\begin{proof}
Let $ U $ be an arbitrary open subset of $ \pi_1^{\mathrm{qtop}}(X, x_0) $ and take $ [g] \in U $. We show that $ [g]\pi_1^{gc}(X, x_0) \subseteq U $ which implies that $ U $ is an open subset of $\pi_1^{\mathrm{gcov}}(X, x_0)$.

Clearly, the trivial element of $ \pi_1(X, x_0) $, $ [c_{x_0}] $, belongs to the coset $ [g^{-1}]U $ and $ [g^{-1}]U $ is an open subset of $\pi_1^{\mathrm{qtop}}(X, x_0)$ since it is a quasitopological group. Then  by \cite[Corollary 2.4]{MTPSG}, $ \overline{[c_{x_0}]} \subseteq [g^{-1}]U $, where $\overline{[c_{x_0}]} = \overline{\pi_1^{sg}(X, x_0)}  $ is the closure of the trivial element in $\pi_1^{\mathrm{qtop}}(X, x_0)$.
 Using the chain of subgroups of the fundamental group which was introduced in \cite[Theorem 2.6]{paper1}, we have $ \pi_1^{gc}(X, x_0) \leq \overline{\pi_1^{sg}(X, x_0)} \subseteq [g^{-1}]U $. Therefore, $ [g]\pi_1^{gc}(X, x_0) \subseteq U $.
\end{proof}

\begin{example} \label{ex319}
Fischer et al. \cite{FiZa} showed that the universal path space of Hawaiian earring, $ HE $, is a generalized covering space. It implies that the trivial subgroup of $ \pi_1(HE, 0) $ is a generalized covering subgroup, i.e. $ \pi_1^{\mathrm{gcov}}(HE, 0) $ is discrete, where we know that $ \pi_1^{\mathrm{qtop}}(HE, 0) $ is not discrete. Moreover, it can easily conclude from Remark \ref{re315} that $ \pi_1^{\mathrm{wh}}(HE, 0) $ also is not discrete. Then, 
\[
 \pi_1^{\mathrm{qtop}}(HE, 0) \subsetneq  \pi_1^{\mathrm{wh}}(HE, 0) \subsetneq \pi_1^{\mathrm{gcov}}(HE, 0).
 \] 
 \end{example}
 
 \begin{example} \label{ex323}
 It was shown in \cite[Example 3.11]{paper1} that the \textit{Harmonic Archipelago}, $ HA $, dose not admit any generalized covering space except the trivial covering. Thus $ \pi_1^{\mathrm{gcov}}(HA, b) $ is trivial, but $ \pi_1^{\mathrm{wh}}(HA, b) $ is discrete where $ b \in HA $ is a non canonical based point. Therefore, the whisker topology and the generalized covering topology may not compare, in general. Moreover, $ \pi_1^{\mathrm{Span}}(HA, b) $ and $ \pi_1^{\mathrm{pSpan}}(HA, b) $ both are trivial since they, unlike the whisker topology, is independent of the choice of the base point.
\end{example}

Recall from \cite{paper3} that since $ HE $ is not a $ SLT $ at $ 0 $ space, then $ \pi_1^{\mathrm{qtop}}(HE, 0)  $ and $ \pi_1^{\mathrm{wh}}(HE, 0) $ are not equal. This fact together with Examples \ref{ex314} and \ref{ex319} implies that all mentioned topologies on the fundamental group of $ HE $ are not equal. Therefore, each of the following topologies is strictly finer than the previous one.
 $$\pi_1^{\mathrm{Span}}(HE, 0) \subsetneq \pi_1^{\mathrm{pSpan}}(HE, 0)  \subsetneq \pi_1^{\mathrm{qtop}}(HE, 0)  \subsetneq \pi_1^{\mathrm{wh}}(HE, 0) \subsetneq \pi_1^{\mathrm{gcov}}(HE, 0).$$

Moreover, since $ \pi_1^{\mathrm{qtop}}(HE, 0) $ is not a topological group, then by \cite[Lemma 41]{BrazFab} the canonical isomorphism $ \varphi: \pi_1^{\mathrm{qtop}}(HE,0) \times \pi_1^{\mathrm{qtop}}(HE,0) \rightarrow \pi_1^{\mathrm{qtop}}(HE \times HE,(0,0)) $ is not continuous, while by Proposition \ref{pr26} it is a homeomorphism  for any of the other topologies mentioned above.

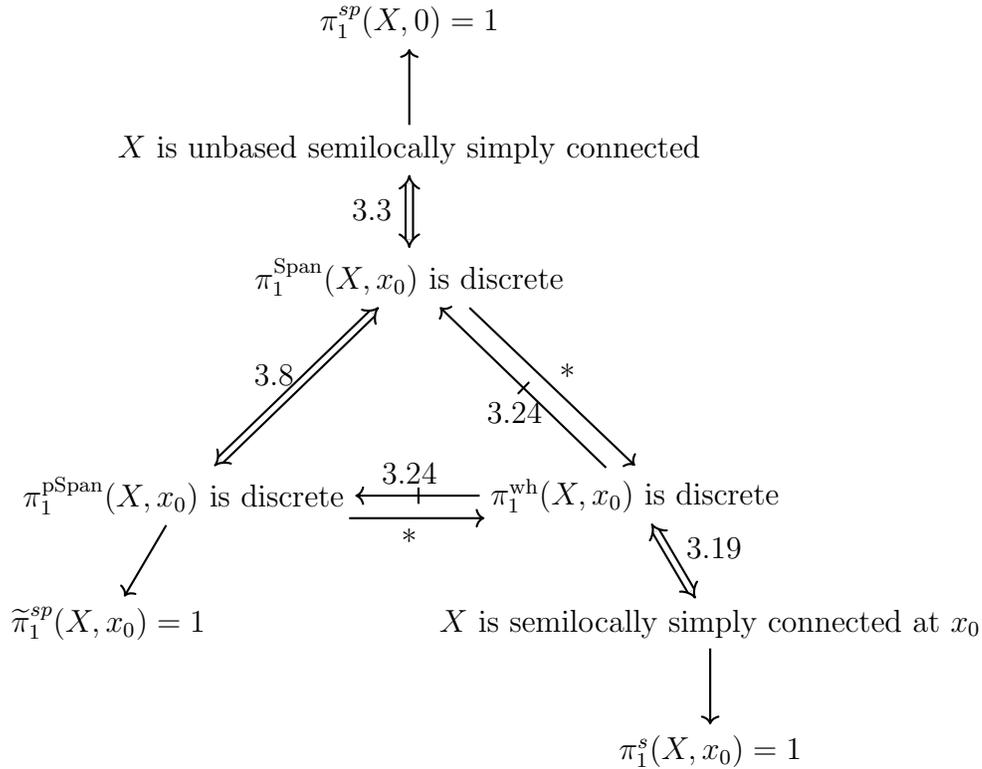
\begin{tikzpicture}
\node (a) at (0,0) {$\pi_1^{sp} (X,0)=1$};
\node (b) at (0,-1.7) {$X$ is unbased semilocally simply connected};
\node (c) at (0,-3.4) {$\pi_1^{\text{Span}} (X,x_0)$ is discrete};
\node (d) at (3,-6.3) {$\pi_1^{\text{wh}} (X,x_0)$ is discrete};
\node (e) at (-3,-6.3) {$\pi_1^{\text{pSpan}} (X,x_0)$ is discrete};
\node (f) at (4,-8) {$X$ is semilocally simply connected at $x_0$};
\node (g) at (-4,-8) {$\widetilde\pi_1^{sp} (X,x_0)=1$};
\node (h) at (4,-9.7) {$\pi_1^s (X,x_0)=1$};

\draw [thick,->] (b) -- (a);
\draw [arrows={Implies}-{Implies},double,double distance=2pt, thick] (b) -- (c);
\draw [thick,->,decoration={
    markings,
    mark=at position 0.5 with {\arrow{|}}},postaction=decorate](d) -- (c);
\draw [thick,->] (.8,-3.8) -- (3,-5.9);
\draw [thick,->,decoration={
    markings,
    mark=at position 0.5 with {\arrow{|}}},postaction=decorate](d) -- (e);
    \draw [thick,->] (-.8,-6.6) -- (1,-6.6);
    
\draw [arrows={Implies}-{Implies},double,double distance=2pt,thick] (d) -- (f);
\draw [arrows={Implies}-{Implies},double,double distance=2pt,thick] (e) -- (c);

\draw [thick,->] (f) -- (h);
\draw [thick,->] (e) -- (g);

\node at (4.05,-7) {\ref{pr318}};
\node at (-.5,-2.5) {\ref{pr33}};
\node at (-1.8,-4.7) {\ref{co38}};
\node at (0,-6) {\ref{ex323}};
\node at (0,-6.9) {*};
\node at (1.4,-5.2) {\ref{ex323}};
\node at (2.1,-4.7) {*};
\end{tikzpicture}

\begin{figure}[h!]
$$Diagram: Discreteness \ of \ some \ subgroup \ topologies \ on \ the \ fundamental \ group$$
$$ (X \ is  \ locally \ path \ connected)$$
\end{figure}

\section*{Reference}

\bibliography{mybibfile}




\end{document}